\documentclass[final]{zzz}
\usepackage{mathrsfs}
\usepackage{color}
\usepackage{tikz}
\usepackage{graphicx}
\usepackage{rotating}
\usepackage[pdftex]{hyperref}
\usepackage{soul}

\DeclareMathAlphabet{\mathpzc}{OT1}{pzc}{m}{it}

\hypersetup{
   colorlinks = true,
   urlcolor = blue,
   linkcolor = red,
   citecolor = green
}

\definecolor{Mycolor2}{HTML}{e85d04}
\sethlcolor{black}

\newcommand{\hyp}[5]{\,\mbox{}_{#1}F_{#2}\!\left(
  \genfrac{}{}{0pt}{}{#3}{#4};#5\right)}

\newcommand{\KdF}[5]{
F_{{#1}}^{{#2}}\!\left(
  \genfrac{}{}{0pt}{}{{#4}}{{#3}};#5\right)}

\def\cprime{$'$}

\newtheorem{thm}{Theorem}[section]
\newtheorem{cor}[thm]{Corollary}

\newtheorem{rem}[thm]{Remark}

\makeatletter
\def\eqnarray{\stepcounter{equation}\let\@currentlabel=\theequation
\global\@eqnswtrue
\tabskip\@centering\let\\=\@eqncr
$$\halign to \displaywidth\bgroup\hfil\global\@eqcnt\z@
  $\displaystyle\tabskip\z@{##}$&\global\@eqcnt\@ne
  \hfil$\displaystyle{{}##{}}$\hfil
  &\global\@eqcnt\tw@ $\displaystyle{##}$\hfil
  \tabskip\@centering&\llap{##}\tabskip\z@\cr}

\def\endeqnarray{\@@eqncr\egroup
      \global\advance\c@equation\m@ne$$\global\@ignoretrue}

\def\@yeqncr{\@ifnextchar [{\@xeqncr}{\@xeqncr[5pt]}}
\makeatother

\newcommand{\dd}{{\mathrm d}}
\newcommand{\expe}{{\mathrm e}}

\newcommand{\ba}{{\hspace{-0.02cm}\boldsymbol \alpha}}

\newcommand{\bm}{{\bf m}}

\newcommand{\N}{{\mathbb N}}

\newcommand{\Z}{\mathbb{Z}} 
\newcommand{\R}{\mathbb{R}} 
\newcommand{\C}{\mathbb{C}} 
\newcommand{\CC}{\mathbb{C}} 
\newcommand{\No}{\mathbb{N}_0} 

\newcommand{\pFq}[5]{  {{}_{#1}F_{#2}}\left(\! \genfrac..{0pt}{}{#3}{#4};#5\!\right)}

\newcommand{\topt}[2]{\left\{\substack{{#1}\\{#2}}\right\}}

\newcommand{\topr}[4]{\left\{\substack{{#1}\\{#2}\\{#3}\\{#4}}\right\}}
\newcommand{\topq}[5]{\left\{\substack{{#1}\\{#2}\\{#3}\\{#4}\\{#5}}\right\}}

\newcommand{\topss}[3]{\left\{\substack{{#1}\\[-0.12cm]{#2}\\{#3}}\right\}}

\begin{document}

\renewcommand{\PaperNumber}{***}

\FirstPageHeading

\ShortArticleName{Contiguous relations for linearization coefficients of orthogonal polynomials}

\ArticleName{Two-dimensional contiguous relations for linearization coefficients of orthogonal polynomials in the Askey-scheme}

\Author{Howard S.~Cohl\,$^{\ast}\footnote{Working remotely in Mission Viejo, California}\!\!\ $
and Lisa Ritter\,$^\ast\!\!$}

\AuthorNameForHeading{H.~S.~Cohl, 
L.~Ritter}

\Address{$^\ast$~Applied and Computational Mathematics Division,
National Institute of Standards and Technology,
Gaithersburg, MD 20899-8910, USA
} 
\EmailD{howard.cohl@nist.gov, lisa.ritter@nist.gov}
\URLaddressD{
\href{http://www.nist.gov/itl/math/msg/howard-s-cohl.cfm}
{http://www.nist.gov/itl/math/msg/howard-s-cohl.cfm}
}





\ArticleDates{Received ???, in final form ????; Published online ????}

\Abstract{We produce two-dimensional
contiguous relations for generalized
hypergeometric functions
by starting with linearization
coefficients for some continuous generalized hypergeometric orthogonal 
polynomials in the Askey-scheme.
}


\section{Introduction}

In this paper we produce two-dimensional contiguous relations for linearization coefficients of hypergeometric orthogonal polynomials in the Askey-scheme. These contiguous relations are derived by using integrals of products of these orthogonal polynomials. These integrals are related to linearization coefficients for the polynomials. The idea for the two-dimensional contiguous relations goes back to a
paper by Ismail, Kasraoui \& Zeng (2013) \cite{IsmailKasraouiZeng2013}.
In this paper we re-derive the general expression for the
two-dimensional contiguous relations and then apply this relation to several specific examples, namely for the linearization of a product of two and three Gegenbauer polynomials, a product of two and three Hermite polynomials, linearization of two Jacobi polynomials and for a product of two unscaled and two scaled Laguerre polynomials.

\section{Preliminaries}
We adopt the following list conventions as follows. 
Within a list of items, we define
\[
a+\topss{x_1}{\vdots}{x_n}:=\{a+x_1,\ldots,a+x_n\},
\]
and when $\pm$ is used within a list of values,
we define $\pm a := \{a,-a \}.$
Let $z\in \C$, $n,k \in \No$ unless otherwise stated. The definition that we use
for the Pochhammer symbol is given by
\begin{eqnarray}
&&\hspace{-5.5cm}(z)_n := (z)(z+1)\cdots(z+n-1),\ (z)_0:=1,\ z\in\C, \label{iden:poch}\\
&&\hspace{-5.5cm}(z_1,\ldots,z_k)_n:=(z_1)_n\cdots(z_k)_n. \label{iden:poch1}
\end{eqnarray}
We will also adopt the following compact notation for the minimum and maximum of any two integers,
$m,n\in\Z$, 
\[
m\!\vee\!n:=\max(m,n),
\quad
m\!\wedge\!n:=\min(m,n).
\]

Define the generalized hypergeometric
series \cite[Chapter 16]{NIST:DLMF}
\begin{equation}
\hyp{r}{s}{a_1,\ldots,a_r}{b_1,\ldots,b_s}{x}
=\sum_{n=0}^\infty
\frac{{(}a_1,\ldots,a_r{)_n}}
{{(}b_1,\ldots,b_s{)_n}}\frac{x^n}{n!},
\end{equation}
and the Kamp\'e de F\'eriet double
hypergeometric series
\cite[(28)]{SriKarl}
\begin{eqnarray}
&&\hspace{-0.7cm}\KdF{l:m;n}{p:q;k}{\alpha_1,\ldots,\alpha_l:\beta_1,\ldots,\beta_m;\gamma_1,\ldots,\gamma_n}{a_1,\ldots,a_p:b_1,\ldots,b_q;c_1,\ldots,c_k}{x,y}
\nonumber\\
&&\hspace{3.5cm}=
\sum_{r,s=0}^\infty
\frac{(a_1,\ldots,a_p)_{r+s}(b_1,\ldots,b_q)_{r}(c_1,\ldots,c_k)_s}
{(\alpha_1,\ldots,\alpha_l)_{r+s}(\beta_1,\ldots,\beta_m)_{r}(\gamma_1,\ldots,\gamma_n)_s}
\frac{x^r}{r!}\frac{y^s}{s!}.
\label{KdF}
\end{eqnarray}
\section{{Two-dimensional contiguous relations which correspond to the linearization of a product of orthogonal polynomials}}

In this section we derive the general relation
for two-dimensional contiguous relations for linearization coefficients
of unscaled orthogonal polynomials, namely Theorem \ref{THMAAA}
below. 
This theorem is similar to \cite[Theorem 2.1]{IsmailKasraouiZeng2013}, albeit for unscaled polynomials $p_n(x)$.  In the remainder of this paper, we compute explicitly the linearization coefficients of the products of two and three orthogonal polynomials.  From this we show examples of these contiguous relations for certain continuous hypergeometric orthogonal polynomials in the Askey scheme. However, in 
\ref{scaledLag} we extend this theorem for scaled Laguerre polynomials. The general result presented in 
\cite[Theorem 2.1]{IsmailKasraouiZeng2013} is for
integrals of products of 
scaled orthogonal polynomials $p_n(\lambda x)$, 
for some $\lambda\in\CC$.

\medskip

Consider a sequence of continuous orthogonal polynomials 
$p_n(x)$,
$m,n\in\N_0$, 
$x\in\C$,
which satisfy the orthogonality relation
\begin{equation}
\int_{\cal I} p_m(x;\ba) p_n(x;\ba) \,\dd\mu =  h_n(\ba) \delta_{m,n},
\label{Orthog}
\end{equation}
where ${\cal I}$ is the support of the measure $\mu$ and ${\ba}$ is a set of parameters.
Orthogonal polynomials satisfy the following 
three-term recurrence relation, with the assumption $p_{-1}(x;{\ba}):=0$,
\begin{equation}
p_{n+1}(x;{\ba})=(A_n x+B_n) p_n(x;{\ba})-C_np_{n-1}(x;{\ba}).
\label{recur}
\end{equation}
The linearization coefficients ${\sf a}_{k,{\bf n}}:={\sf a}_{\bf m}$,
where 
${\bf n}:=\{n_1,\ldots,n_{N}\}
\in\N_0^N$, 
$N\ge 2$, 
$k\in\N_0$,
${\bf m}:=\{k\}\cup{\bf n}$,
are defined using 
\begin{equation}
p_{n_1}(x;{\ba})\cdots p_{n_N}(x;{\ba})
=\sum_{k=0}^{n_1+\cdots+n_N} {\sf a}_{{\bf m}}\,p_k(x;{\ba}).
\label{multlin}
\end{equation}
Note that even though it is true that
$k\in\{0,\ldots,n_1+\cdots+n_N\}$, 
it may be that $k$ has a more restricted range depending on the specific
orthogonal polynomials involved.

\medskip
Now consider the integral over $N+1$ orthogonal polynomials
${\sf P}({\bf m};\ba)\in\R$, where
$\bm:=\{n_1,\ldots,n_{N+1}\}$ by 
\begin{equation}
{\sf P}(\bm;\ba)
:=\int_{\cal I} p_{n_1}(x;{\ba})\cdots p_{n_{N}}(x;{\ba}) p_{n_{N+1}}(x;{\ba})\,\dd\mu.
\label{deff}
\end{equation}
We can formally see that integral is 
associated with the 
linearization of a product of
$N$ orthogonal polynomials as follows.
Without loss of generality
choose 
\begin{equation}
n_1\le n_2 \le \cdots\le n_{N}\le n_{N+1}.
\end{equation}
Then after substituting \eqref{multlin} in \eqref{deff} and using \eqref{Orthog}, one obtains
\begin{equation}
{\sf P}(\bm;\ba)=
\int_{\cal I} 
\sum_{k=0}^{n_1+\cdots+n_N}
{\sf a}_{\bm}\,
p_{n_{N+1}}(x;\ba)
p_k(x;\ba)\,\dd\mu
=h_{N+1}(\ba){\sf a}_{\bm},
\label{fa}
\end{equation}
and since overlap in the orthogonality 
only occurs if $n_{N+1}\le n_1+\cdots +n_N$,
the integral ${\sf P}(\bm;\ba)$ will vanish otherwise.



\medskip
Now define 
\[
{\sf P}_j^{\pm}:=
{\sf P}_j^{\pm}({\bf m};\ba):={\sf P}(n_1,\ldots,n_{j-1},n_{j}\pm 1,n_{j+1},\ldots,n_{N+1};\ba).
\]

\medskip
\begin{thm}
Let $N,j,k\in\N_0$, $1\le j<k\le N+1$, $n_j,n_k\in{\bf m}$, $x\in{\cal I}$, 
$A_{n_j},B_{n_j},C_{n_j}$ satisfy the
three-term recurrence relation 
(\ref{recur}).
Then the definite integral ${\sf P}({\bm};\ba)$ corresponding to a product of $N+1$ 
orthogonal polynomials defined by (\ref{deff}), satisfies the 
following sequence of $\binom{N+1}{2}$ contiguous relations
\[
(B_{n_j}A_{n_k}-A_{n_j}B_{n_k}){\sf P}-A_{n_k}{\sf P}^{+}_{j}-C_{n_j}A_{n_k}{\sf P}^{-}_{j}+A_{n_j}{\sf P}^{+}_{k}+C_{n_k}A_{n_j}{\sf P}^{-}_{k}=0.
\]
\label{THMAAA}
\end{thm}
\begin{proof}
Choose a quantum number $m:=n_{j}\in{\bf m}$, and then evaluate ${\sf P}^{+}_j$ 
using (\ref{deff}). Then use (\ref{recur}) to replace $p_{m+1}(x)$ with
$(A_{m}x+B_{m})p_{m}(x)-C_{m}p_{m-1}(x)$ in the integrand. Now choose a different
quantum number 
$n:=n_{k}\in{\bf m}$ and replace $xp_n(x)$ in the resulting expression with 
\begin{equation}
xp_n(x)=\frac{1}{A_n}\left(p_{n+1}-B_np_n(x)+C_np_{n-1}\right).
\label{inforecurrel}
\end{equation}
Multiplying both sides of the resulting expression by $A_{n}$ produces 
a five term two-dimensional contiguous relation for ${\sf P}$ involving 
${\sf P}_j^{\pm}$ and ${\sf P}_k^{\pm}$.
Repeating this process for all $\binom{N+1}{2}$ unique combinations of 
quantum numbers in ${\bf m}$ completes the proof.
\end{proof}

\begin{rem}
For a linearization of a product of $N$ orthogonal polynomials, choose $j\ne k\in\{1,\ldots,N\}$.
The three-term recurrence relation \eqref{inforecurrel} is symmetric under permutation of $n_j$ and $n_k$.  Hence each unique contiguous relation is the result of choosing two quantum numbers from $N+1$ possibilities
and therefore one will obtain a sequence of $\binom{n+1}{2}$ contiguous relations for the 
definite integral of a product of $N+1$ orthogonal 
polynomials.
\end{rem}
\noindent In the remainder of the paper, we compute examples of these contiguous relations for certain continuous hypergeometric orthogonal polynomials in the Askey scheme.

\section{Gegenbauer polynomials}

The Gegenbauer (or ultraspherical) polynomials can be defined  as \cite[(9.8.19)]{Koekoeketal}
\begin{equation}
C_n^\lambda (x) = \frac{(2\lambda)_n}{n!} \pFq{2}{1}{-n,n+2\lambda}{\lambda 
+\frac12}{\frac{1-x}{2}},
\label{ortho:gegen}
\end{equation}
where $n\in\N_0$, 
$\lambda \in \C\setminus \{0\}$, 
$x\in\C$.
Gegenbauer polynomials are orthogonal on $x\in(-1,1)$, with orthogonality relation
\cite[(9.8.20)]{Koekoeketal}
\begin{equation}
\int_{-1}^1 C_m^\lambda(x)C_n^\lambda(x) (1-x^2)^{\lambda-\frac12}\,\dd x
=\frac{\pi\,\Gamma(2\lambda+n)\delta_{m,n}}{2^{2\lambda-1}(n+\lambda)n!\Gamma(\lambda)^2}
=: h_n^\lambda\delta_{m,n},
\label{orthoggeg}
\end{equation}
where $\lambda \in (-\tfrac12,\infty)\backslash \{0\}$. 
The three-term recurrence relation for Gegenbauer polynomials is described via 
\eqref{recur} with \cite[Table 18.9.1]{NIST:DLMF}
\begin{gather} 
 A_n = \frac{2(n + \lambda)}{n+1},\quad \quad B_n = 0, \quad \quad C_n = \frac{n + 2 \lambda -1}{n+1}. \label{recurGeg}
\end{gather}

\subsection{Linearization of a product of two Gegenbauer polynomials}

Gegenbauer polynomials 
have a linearization formula
{for a product of two Gegenbauer polynomials} given by \cite[(18.18.22)]{NIST:DLMF}
\begin{equation}
C^\lambda_m (x) C^\lambda_n (x) = \sum^{m}_{k=0} B^{\lambda}_{k,m,n} 
C^{\lambda}_{m+n-2k} (x),  
\label{lin:gegen}
\end{equation}
where $m,n \in \mathbb{N}_0$, without loss of 
generality $n\geq m$, and 
\begin{equation} \label{geg2}
B^{\lambda}_{k,m,n} := \frac{(m\!+\!n\!+\! \lambda \!-\! 2k)(m\!+\!n\!-\!2k)!(\lambda)_k 
(\lambda)_{m-k} (\lambda)_{n-k} (2\lambda)_{m+n-k}}{(m\!+\!n\!+\!\lambda \!-\!k) 
k! (m\!-\!k)!(n\!-\!k)! (\lambda)_{m+n-k} (2\lambda)_{m+n-2k}}.
\end{equation}
Using (\ref{orthoggeg}), we see that  
(\ref{lin:gegen}) {is equivalent to the following integral of a product of three Gegenbauer polynomials}
\begin{equation}
{\sf C}(k,m,n;\lambda):=\int_{-1}^1 
C_m^\lambda(x)
C_n^\lambda(x)
C_{m+n-2k}^\lambda(x)
(1-x^2)^{\lambda-\frac12} \,\dd x =  h_{m+n-2k}^\lambda B_{k,m,n}^\lambda.
\label{prod2gegint}
\end{equation}

\begin{rem}
Using the recurrence relations given by \eqref{recurGeg} with Theorem \ref{THMAAA} one obtains the following contiguous relations:
\begin{eqnarray}
&&(m\!+\!\lambda)(p+1){\sf C}_1^{+}\!+\!(p\!+\!2\lambda \!-\!1)
(m\!+\!\lambda) {\sf C}_1^{-}\!=\!( p\!+\!\lambda)(m+1) {\sf C}_2^{+}\!+\!(p+ 
\lambda)(m+2\lambda \!-\!1){\sf C}_2^{-} , \label{3gcont1} \\
 &&(n\!+\!\lambda)(m+1) {\sf C}_2^{+}\!+\!(m\!+\!2 \lambda \!-\!1)
(n +\lambda) {\sf C}_2^{-}\!=\!( m\!+\!\lambda)(n+1) {\sf C}_3^{+}\!+\!
(m+ \lambda)(n+2\lambda \!-\!1){\sf C}_3^{-}, \label{3gcont2} \\
&&(n\!+\!\lambda)(p+1) {\sf C}_1^{+}\!+\!(p\!+\!2 \lambda \!-\!1)
(n +\lambda) {\sf C}_1^{-}\!=\!( p\!+\!\lambda)(n+1) {\sf C}_3^{+}\!+\!(p+ \lambda)(n+2\lambda \!-\!1){\sf C}_3^{-}. \label{3gcont3}
\end{eqnarray}
To use these with \eqref{prod2gegint}, we can alter it by taking $p = m+n-2k$, to get the following:
\begin{equation} \label{prod2gegint2}
    {\sf C}(p,m,n;\lambda) := h_p^\lambda B_{\frac12(m+n-p),m,n}^\lambda.
\end{equation}
For these contiguous relations 
\eqref{3gcont1}-\eqref{3gcont3} to be satisfied, there are conditions on the parameters.  Namely, $p \in \{n-m,...n+m\}$ and $m+n-p\pm1$ must be even.
These can be combined with \eqref{prod2gegint2}, where the numbers $1$, $2$ and $3$ correspond to the parameters $p$, $m$, and $n$ respectively.  
This produces the following:
\begin{eqnarray}
&&\label{B1t}\hspace{-0.3cm}   \frac{(\lambda+m)   (2 \lambda +p)}{(\lambda+p+1)} B_{\frac12(m+n-p-1),m,n} +  \frac{p   ( \lambda+m)}{(\lambda+p\!-\!1)} B_{\frac12(m+n-p+1),m,n} \nonumber \\
&& \hspace{1.35cm} = (m+1)   B_{\frac12(m+n-p+1),m+1,n} + (m+2\lambda -1)  B_{\frac12(m+n-p-1),m-1,n},   \\
&&\hspace{-0.3cm}(n + \lambda)(m+1)   B_{\frac12(m+n-p+1),m+1,n} +(m + 2 \lambda -1)(n + \lambda)   B_{\frac12(m+n-p-1),m-1,n} \nonumber \\
&& \hspace{0.85cm}= ( m + \lambda)(n+1)   B_{\frac12(m+n-p+1),m,n+1} + (m+ \lambda)(n+2\lambda -1)  B_{\frac12(m+n-p-1),m,n-1}, \\
&&\hspace{-0.3cm}  \frac{(\lambda+n)(2\lambda+p)}{(\lambda+p+1)} B_{\frac12(m+n-p-1),m,n} +  \frac{p(\lambda+n)}{(\lambda+p-1)} B_{\frac12(m+n-p+1),m,n} \nonumber \\
&& \hspace{1.35cm}= (n+1)   B_{\frac12(m+n-p+1),m,n+1} + (n+2\lambda -1)  B_{\frac12(m+n-p-1),m,n-1}.
\end{eqnarray}
Take for instance for \eqref{B1t} using  
 the identities of \eqref{geg2} and \eqref{orthoggeg}.
Then it reduces to
\begin{align}
    &\hspace{-2.4em}(m\!+\!\lambda)(p+1){\sf C}_1^{+}\!+\!(p\!+\!2\lambda \!-\!1)
(m\!+\!\lambda) {\sf C}_1^{-}\!-\!( p\!+\!\lambda)(m+1) {\sf C}_2^{+}\!-\!(p+ 
\lambda)(m+2\lambda \!-\!1){\sf C}_2^{-} \nonumber\\
& \hspace{-1.9em} = \frac{2p!(\lambda)_{\frac12 (n+m-p-1})(\lambda)_{\frac12 (m+p-n-1)}(\lambda)_{\frac12 (n+p-m-1)}(2 \lambda)_{\frac12 (n+m+p-1)}}{\Gamma(\frac12 (n+m-p+1))\Gamma(\frac12 (m-n+p+1))\Gamma(\frac12 (n-m+p+1))(\lambda)_{\frac12 (n+m+p-1)}(2\lambda)_{p-1}} \nonumber\\
& \hspace{-0.1em} \times \left( \frac{(p\!+\!1)(\lambda \!+\!m)(2\lambda \!+\!m\!+\!p\!-\!n\!-\!1)(2\lambda \!+\!n\!+\!p\!-\!m\!-\!1)(4 \lambda \!+\!n\!+\!m\!+\!p\!-\!1)}{(m\!+\!p\!-\!n\!+\!1)(n\!+\!p\!-\!m\!+\!1)(2\lambda\!+\!p\!-\!1)(2 \lambda\!+\!m\!+\!n\!+\!p\!-\!1)(2 \lambda \!+\!m\!+\!n\!+\!p\!+\!1)} \right. \nonumber\\
& \hspace{2.4em} \!-\! \frac{(m\!+\!1)(p\!+\!\lambda)(2\lambda \!+\!m\!+\!n\!-\!p\!-\!1)(2\lambda\!+\!m\!+\!p\!-\!n\!-\!1)(4\lambda\!+\!m\!+\!n\!+\!p\!-\!1)}{(m\!+\!n\!-\!p\!+\!1)(m\!+\!p\!-\!n\!+\!1)(2\lambda\!+\!p\!-\!1)(2\lambda\!+\!m\!+\!n\!+\!p\!-\!1)(2\lambda\!+\!m\!+\!n\!+\!p\!+\!1)}\nonumber\\
& \hspace{2.4em} \!-\! \frac{(\lambda\!+\!p)(2\lambda\!+\!m\!-\!1)(2\lambda\!+\!n\!+\!p\!-\!m\!-\!1)}{(n\!+\!p\!-\!m\!-\!1)(2\lambda\!+\!p\!-\!1)(2\lambda\!+\!m\!+\!n\!+\!p\!-\!1)} \nonumber\\
& \hspace{2.4em} \!+\! \left. \frac{(\lambda\!+\!m)(2\lambda\!+\!m\!+\!n\!-\!p\!-\!1)}{(m\!+\!n\!-\!p\!+\!1)(2\lambda\!+\!m\!+\!n+p-1)} \right).
\end{align}
Since the rational coefficient multiplying the factorials and Pochhammer symbols vanishes, we can see that this identity is trivially satisfied.
So this is a clear validation of the contiguous relations 
implied by Theorem \ref{THMAAA}.
Note that similar validations can be obtained by using \eqref{3gcont2}, \eqref{3gcont3}, which we leave to the reader.
\end{rem}

%

\subsection{Linearization of a product of three Gegenbauer polynomials}

\noindent Now we present the linearization formula for a product of three Gegenbauer polynomials.

\begin{thm}
\label{prodthreegegt}
Let $p,m,n\in\N_0$ and without loss of generality $p\le m\le n$, $\lambda\in(-\frac12,\infty)\setminus\{0\}$,
$x\in\C$.
Then 
\begin{equation}
C_p^\lambda(x)
C_m^\lambda(x)
C_n^\lambda(x)
=
\sum_{k=0}^{\lfloor\frac{p+m+n}{2}\rfloor}
{\sf F}_{k,p,m,n}^\lambda C_{p+m+n-2k}^\lambda(x),
\label{prodthreegeg}
\end{equation}
where
\begin{equation}
{\sf F}_{k,p,m,n}^\lambda:=
\left\{
\begin{array}{lll}
{\sf D}_{k,p,m,n}^\lambda, \quad&\mbox{if}&\quad 0\le k \le p-1,\\[0.1cm]
{\sf E}_{k,p,m,n}^\lambda, \quad&\mbox{if}&\quad p\le k \le \lfloor \frac{p+m+n}{2}\rfloor,
\end{array}
\right.
\end{equation}
\small
\begin{eqnarray}
&&\hspace{-0.5cm}{\sf D}_{k,p,m,n}^\lambda\!:=
\frac{
(m\!+\!n)!(\lambda)_m(\lambda)_n
(\lambda\!+\!m\!+\!n\!+\!p\!-\!2k)(m\!+\!n\!+\!p\!-\!2k)!(\lambda)_k(\lambda)_{p-k}(\lambda)_{m+n-k}(2\lambda)_{m+n+p-k}
}{
m!n!(\lambda)_{m\!+\!n}
(\lambda\!+\!m\!+\!n\!+\!p\!-\!k)k!(p-k)!(m\!+\!n-k)!(\lambda)_{m+n+p-k}(2\lambda)_{m+n+p-2k}
}\nonumber\\
&&\hspace{2cm}\times\!
\hyp{11}{10}{
\topt{\lambda}{\lambda+p-k},
\topr{-k}
{-m}
{-n}
{-m-n+k}\!,
\topt{-\lambda-m-n}{-\lambda-m-n-p+k},
\!\frac{-\lambda+2-m-n}{2},
\frac{-2\lambda+\topt{1}{2}-m-n}{2}
}{
1\!+\!p\!-\!k,
\topr{-\lambda+1-k}{-\lambda+1-m}{-\lambda+1-n}{-\lambda+1-m-n+k}\!,
\topt{-2\lambda+1-m-n}{-2\lambda+1-m-n-p+k}\!,
\frac{-\lambda-m-n}{2},
\!\frac{\topt{0}{1}-m-n}{2}
}{1}\!,
\end{eqnarray}
\begin{eqnarray}
&&\hspace{-0.3cm}{\sf E}_{k,p,m,n}^\lambda\!:=
\frac{
(\lambda\!+\!m\!+\!n\!+\!p\!-\!2k)(m\!+\!n\!+\!2p\!-\!2k)!(\lambda)_p(\lambda)_{k-p}(\lambda)_{m+p-k}(\lambda)_{n+p-k}
(\lambda)_{m+n+p-2k}(2\lambda)_{m+n+p-k}
}{
(\lambda\!+\!m\!+\!n\!+\!p\!-\!k)p!(k\!-\!p)!(m\!+\!p\!-\!k)!(n\!+\!p\!-\!k)!(\lambda)_{m+n+p-k}(\lambda)_{m+n+2p-2k}(2\lambda)_{m+n+p-2k}
}\nonumber\\
\times\!
&&
\hyp{11}{10}{
\topt{\lambda}{\lambda+k-p}\!,
\topq{-p}{-m-p+k}{-n-p+k}{-m-n-p+2k}\!,
\topt{-\lambda-m-n-p+k}{-\lambda-m-n-2p+2k}\!,
\frac{-\lambda+2k-2p-m-n+2}{2},
\frac{-2\lambda+\topt{1}{2}+2k-2p-m-n}{2}
}{
1\!+\!k\!-\!p,
\topr{-\lambda+1-p}{-\lambda+1-m-p+k}{-\lambda+1-n-p+k}{-\lambda+1-m-n-p+2k}\!,
\topt{-2\lambda+1-m-n-p+k}{-2\lambda+1-m-n-2p+{2}k}\!,
\frac{-\lambda+2k-2p-m-n}{2},
\!\frac{\topt{0}{1}+2k-2p-m-n}{2}}{1}\!.
\end{eqnarray}
\normalsize
\end{thm}
\begin{proof}
Consider (\ref{prodthreegeg}), where 
${\sf F}_{k,p,m,n}^\lambda$ is to be determined.  
Using (\ref{prod2gegint}), one obtains {the following integral of a product of four Gegenbauer polynomials which is equivalent to the linearization of a product of three Gegenbauer polynomials, namely}
\begin{equation}
{\sf F}_{k,p,m,n}^\lambda=\frac{1}{ h_{m+n+p-2k}^\lambda}
\int_{-1}^1 
C_p^\lambda(x)
C_m^\lambda(x)
C_n^\lambda(x)
C_{p+m+n-2k}^\lambda(x)
(1-x^2)^{\lambda-\frac12}\,\dd x.
\label{intexprcoefsfF}
\end{equation}
Now use (\ref{lin:gegen}) to write the product of Gegenbauer polynomials
of degree $m$ and $n$ as a single sum over $l\in\N_0$, $0\le l\le m$.
This converts (\ref{intexprcoefsfF}) into a single sum of an integral of a 
product of three Gegenbauer polynomials whose terms can be evaluated using 
(\ref{prod2gegint}).  After avoiding the factor $l+p-k$ becoming negative (in which case 
${\sf F}_{k,p,m,n}^\lambda$ vanishes), we obtain
\[
{\sf F}_{k,p,m,n}^\lambda=\sum_{l=\max(0,k-p)}^m 
B_{l,m,n}^\lambda
B_{k-l,p,m+n-2l}^\lambda,
\]
which breaks the linearization formula into two regions depending on $k$, namely
\[
C_p^\lambda(x)
C_m^\lambda(x)
C_n^\lambda(x)=
\sum_{k=0}^{p-1}
C_{p+m+n-2k}^\lambda(x) {\sf D}_{k,p,l,m}^\lambda
+
\sum_{k=p}^{\lfloor\frac{p+m+n}{2} \rfloor}
C_{p+m+n-2k}^\lambda(x) {\sf E}_{k,p,l,m}^\lambda,
\]
where
\[
{\sf D}_{k,p,m,n}^\lambda:=\sum_{l=0}^m 
B_{l,m,n}^\lambda
B_{k-l,p,m+n-2l}^\lambda,
\quad
{\sf E}_{k,p,m,n}^\lambda:=\sum_{l=0}^{p+m-k}
B_{l+k-p,m,n}^\lambda
B_{p-l,p,m+n+2p-2k-2l}^\lambda.
\]
Factoring the products of linearization coefficients in terms of 
Pochhammer symbols and factorials completes the proof.
\end{proof}

\begin{rem}
Setting $p=0$ in (\ref{prodthreegeg}) straightforwardly produces (\ref{lin:gegen}) since the first sum
vanishes and the ${}_{11}F_{10}(1)$ in the second sum is unity.
\end{rem}

\begin{cor} 
Let $p,m,n\in\N_0$ and without loss of generality $p\le m\le n$, $\lambda\in(-\frac12,\infty)\setminus\{0\}$.
Then 
\begin{equation}
\frac{
(2\lambda)_p
(2\lambda)_m
(2\lambda)_n
}{p!m!n!}=
\frac{(2\lambda)_{p+m+n}}{(p+m+n)!}
\sum_{k=0}^{\lfloor\frac{p+m+n}{2}\rfloor}
{\sf F}_{k,p,m,n}^\lambda
\frac{
\left(\frac{-p-m-n}{2}\right)_k
\left(\frac{-p-m-n+1}{2}\right)_k
}{
\left(\frac{-2\lambda-p-m-n+1}{2}\right)_k
\left(\frac{-2\lambda-p-m-n+2}{2}\right)_k
}.
\label{qto1firstlimit}
\end{equation}
\end{cor}
\begin{proof}
Let $x=1$ in (\ref{prodthreegeg}) and \cite[Table 18.6.1]{NIST:DLMF}
$C_n^\lambda(1)=(2\lambda)_n/n!$.
\end{proof}

Define the following
integral of the product of four
Gegenbauer polynomials
\begin{equation}
{\sf C}(p,m,n,l;\lambda):=\int_{-1}^1 
C_p^\lambda(x)
C_m^\lambda(x)
C_n^\lambda(x)
C_{l}^\lambda(x)
(1-x^2)^{\lambda-\frac12}\,\dd x,
\end{equation}
where $l=p+m+n-2k$.
We now present the contiguous relations for the
{integral of a} product of four Gegenbauer polynomials.

\begin{thm}
Let $p,m,n\in\N_0$, and without loss of 
generality $p \le m \le n$, and $l \in \{0,...,p+m+n\}$ 
such that $p+m+n-l \pm 1$ is even. 
Then 
\begin{align}
& \hspace{-0.8cm}  (m\!+\!\lambda)(p\!+\!1)    \,{\sf F}_{\frac12(p+m+n-l+1),p+1,m,n} +(p\!+\!2 \lambda\!-\!1)(m\!+\!\lambda)    \,{\sf F}_{\frac12(p+m+n-l-1),p-1,m,n} \nonumber \\
& = ( p\!+\!\lambda)(m\!+\!1)    \,{\sf F}_{\frac12(p+m+n-l+1),p,m+1,n}\!+\!(p\!+\! \lambda)(m\!+\!2\lambda\!-\!1)   \,{\sf F}_{\frac12(p+m+n-l-1),p,m-1,n},    \\
&\hspace{-0.8cm}(n + \lambda)(p+1)    \,{\sf F}_{\frac12(p+m+n-l+1),p+1,m,n} +(p + 2 \lambda -1)(n + \lambda)    \,{\sf F}_{\frac12(p+m+n-l-1),p-1,m,n} \nonumber \\
&{}= ( p + \lambda)(n+1)    \,{\sf F}_{\frac12(p+m+n-l+1),p,m,n+1}\! +\! (p+ \lambda)(n+2\lambda -1)   \,{\sf F}_{\frac12(p+m+n-l-1),p,m,n-1},  \\
&\hspace{-0.8cm} (p+1)  \,{\sf F}_{\frac12(p+m+n-l+1),p+1,m,n} +(p + 2 \lambda -1)  \,{\sf F}_{\frac12(p+m+n-l-1),p-1,m,n} \nonumber \\
& =   \frac{(2\lambda+l)( p + \lambda)}{(\lambda+l+1)}\,{\sf F}_{\frac12(p+m+n-l-1),p,m,n} +  \frac{l(p+ \lambda)}{(\lambda+l-1)}\,{\sf F}_{\frac12(p+m+n-l+1),p,m,n},   \\
&\hspace{-0.8cm}(n\!+\!\lambda)(m\!+\!1)    \,{\sf F}_{\frac12(p+m+n-l+1),p,m+1,n} \!+\!(m\!+\!2 \lambda -1)(n\!+\!\lambda)    \,{\sf F}_{\frac12(p+m+n-l-1),p,m-1,n} \nonumber \\
& = ( m\!+\!\lambda)(n\!+\!1)    \,{\sf F}_{\frac12(p+m+n-l+1),p,m,n+1}\!+\!(m\!+\! \lambda)(n\!+\!2\lambda -1)   \,{\sf F}_{\frac12(p+m+n-l-1),p,m,n-1},  \\
&\hspace{-0.8cm}(m+1)     \,{\sf F}_{\frac12(p+m+n-l+1),p,m+1,n} +(m + 2 \lambda -1)     \,{\sf F}_{\frac12(p+m+n-l-1),p,m-1,n} \nonumber \\
& =  \frac{(2\lambda+l)( m + \lambda)}{(\lambda+l+1)}\,{\sf F}_{\frac12(p+m+n-l-1),p,m,n} +  \frac{l(m+ \lambda)}{(\lambda+l-1)}\,{\sf F}_{\frac12(p+m+n-l+1),p,m,n},   \\
&\hspace{-0.8cm}(n+1)     \,{\sf F}_{\frac12(p+m+n-l+1),p,m,n+1} +(n + 2 \lambda -1)     \,{\sf F}_{\frac12(p+m+n-l-1),p,m,n-1} \nonumber \\
& =   \frac{(2\lambda+l)( n + \lambda)}{(\lambda+l+1)}\,{\sf F}_{\frac12(p+m+n-l-1),p,m,n} +  \frac{l(n+ \lambda)}{(\lambda+l-1)}\,{\sf F}_{\frac12(p+m+n-l+1),p,m,n}, 
\end{align}
where 
\[
{\sf F}_{\frac12(p+m+n-l\pm 1),p,m,n}^\lambda:=
\left\{
\begin{array}{lll}
{\sf D}_{\frac12(p+m+n-l\pm 1),p,m,n}^\lambda, \quad&\mbox{if}&\quad m+n-p+2 \le l \le m+n+p,\\[0.1cm]
{\sf E}_{\frac12(p+m+n-l\pm 1),p,m,n}^\lambda, \quad&\mbox{if}&\quad 0 \le l\le m+n-p.
\end{array}
\right.
\]
\end{thm}
\begin{proof}
{Aside from} the contiguous relations  \eqref{3gcont1}--\eqref{3gcont3}, there are three more
\begin{eqnarray}
\label{Gegc1}
& \hspace{-0.7cm}  (l + \lambda)(p+1) {\sf C}_1^{+} +(p + 2 \lambda -1)(l + \lambda) {\sf C}_1^{-} = ( p + \lambda)(l+1) {\sf C}_4^{+} + (p+ \lambda)(l+2\lambda -1){\sf C}_4^{-} , \\
\label{Gegc2}
&  \hspace{-0.7cm}  (l + \lambda)(m+1) {\sf C}_2^{+} +(m + 2 \lambda -1)(l + \lambda) {\sf C}_2^{-} = ( m + \lambda)(l+1) {\sf C}_4^{+} + (m+ \lambda)(l+2\lambda -1){\sf C}_4^{-},   \\
\label{Gegc3}
& \hspace{-0.7cm}   (l + \lambda)(n+1) {\sf C}_3^{+} +(n + 2 \lambda -1)(l + \lambda) {\sf C}_3^{-} = ( n + \lambda)(l+1) {\sf C}_4^{+} + (n+ \lambda)(l+2\lambda -1){\sf C}_4^{-},
\end{eqnarray}
because we now have an
integral of a 
product of four orthogonal polynomials and therefore there will be 
$\binom{4}{2}=6$
contiguous relations.
The parameters $p$, $m$, $n$ and $l$ are associated with the subscripts $1$, $2$, $3$ and $4$ respectively, and $l = p+m+n-2k$. Applying this to the six contiguous relations 
\eqref{3gcont1}--\eqref{3gcont3}, \eqref{Gegc1}--\eqref{Gegc3}, and using Theorem \ref{prodthreegegt} completes the proof.
\end{proof}
\section{Hermite polynomials}
The Hermite polynomials can be defined  as \cite[(9.15.1)]{Koekoeketal}
\begin{equation}
H_n(x) = (2x)^n \pFq{2}{0}{-\tfrac12 n,-\tfrac12(n-1)}{-}{-\frac{1}{x^2}},
\label{ortho:herm}
\end{equation}
where $n\in\N_0$, $x\in\C$.
Hermite polynomials are orthogonal on $x\in(-\infty,\infty)$, with orthogonality relation
\cite[(9.15.2)]{Koekoeketal}
\begin{equation}
\int_{-\infty}^\infty H_m(x)H_n(x) \,\expe^{-x^2}\,\dd x
=\sqrt{\pi}\,2^n n!\delta_{m,n}
=:h_n\delta_{m,n}.
\label{orthogherm}
\end{equation}
The recurrence relation (\ref{recur}) for Hermite polynomials is given through
\begin{equation}
A_n=2,\quad B_n=0,\quad C_n=2n.
\label{recurherm}
\end{equation}
The definite integral of a product of $N\in\N_0$, $N\ge 2$, Hermite polynomials is defined by
\[
{\sf H}({\bf n}):=\int_{-\infty}^\infty H_{n_1}(x)\cdots H_{n_{N+1}}(x)\,\expe^{-x^2}\,\dd x,
\]
where ${\bf n}:=\{n_1,\ldots,n_{N+1}\}$.
Note that ${\sf H}({\bf n})$ has a generating function given by \cite[Exercise 4.11]{Ismail:2009:CQO}
\begin{equation}
\exp\left(2\sum_{1\le i< j\le k}t_it_j\right)
=\frac{1}{\sqrt{\pi}}\sum_{n_1,\ldots,n_{k}=0}^\infty \frac{t_1^{n_1}\cdots t_{k}^{n_{k}}}
{n_1!\cdots n_{k}!}{\sf H}({\bf n}).
\label{multhermgenfun}
\end{equation}
\subsection{{Linearization of a product of two Hermite polynomials}}
The Hermite polynomials (\ref{ortho:herm}) 
have a linearization formula given by \cite[(18.18.22)]{NIST:DLMF}
\begin{equation}
H_m (x) H_n (x) = \sum^{m}_{k=0} {\sf  b}_{k,m,n} 
H_{m+n-2k} (x),  
\label{lin:herm}
\end{equation}
where $m,n \in \mathbb{N}_0$, without loss of generality $n\geq m$, and 
\begin{equation*}
{\sf  b}_{k,m,n} := \frac{2^km!n!}{k!(m-k)!(n-k)!}.
\end{equation*}
The definite integral corresponding to the linearization of a product of two Hermite polynomials 
is given by
\begin{equation}
{\sf H}({p,m,n}):=\int_{-\infty}^\infty H_p(x) H_m(x) H_n(x) \,\expe^{-x^2} \,\dd x.
\label{threintherm}
\end{equation}
Using (\ref{lin:herm}) and orthogonality (\ref{orthogherm}), we see that ${\sf H}(p,m,n)$ is given by 
\begin{equation}
\hspace{-0.1cm}{\sf H}(p,m,n)\!=\!\left\{ 
\begin{array}{llc}
0, \quad&\mbox{if}
\ \, p\!>\!n\!+\!m\ \mbox{or}
\ \, m\!>\!n\!+\!p\ \mbox{or}
\ \, n\!>\!m\!+\!p &\\
&
\mbox{or} \ ((p+m+n)\!\!\!\!\mod 2) = 1, &\\
{\displaystyle \frac{\sqrt{\pi}\,m!n!p!2^{\lfloor\frac{p+m+n}{2}\rfloor}}
{
\lfloor\frac{m+n-p}{2}\rfloor !
\lfloor\frac{m+p-n}{2}\rfloor !
\lfloor\frac{n+p-m}{2}\rfloor !
}}, \quad&\mbox{otherwise}.&
\end{array}
\right.
\label{defintherm3}
\end{equation}

\begin{cor}
Let $|t_1|, |t_2|, |t_3|<1$. Then ${\sf H}(p,m,n)$ has the following multilinear 
generating function (\ref{multhermgenfun}), given by
\[
\exp(t_1t_2+t_1t_3+t_2t_3)=\sum_{n,m=0}^\infty
\sum_{p=0}^{m\wedge n} \frac
{2^{p+m\vee n} 
t_1^{2p+|n-m|} 
t_2^mt_3^n 
}
{
(m\!\wedge\!n\!-\!p)!\,
(p\!+\!|n-m|)!\,
p!}.
\]
\end{cor}
\begin{proof}
Starting with (\ref{multhermgenfun}), (\ref{defintherm3}), for a fixed $m,n\in\N_0$,
$p$ is non-vanishing for $|n-m|\le p\le n+m$. Shifting the $p$-index by $|n-m|$
and scaling $p$ by a power of two to remove the remaining vanishing values of  
${\sf H}(p,n,m)$ completes the proof.
\end{proof}

\begin{rem}
\label{rem5.2}
Let $x \in \mathbb{C}$, $p,m,n\in \mathbb{N}_0$, and without loss of generality assume $m \le n$.  Further, let $p \in \{n-m,...,n+m\}$ and $p+m+n \pm 1$ even. 
Combining Theorem \ref{THMAAA} with the coefficients of the recurrence relation given by \eqref{recurherm} yields three contiguous relations for the integral of the product of three Hermite polynomials:
\begin{eqnarray}
&&\hspace{-8.7cm}{\sf H}^{+}_{2}-{\sf H}^{+}_{1}-2p{\sf H}^{-}_{1} +2m{\sf H}^{-}_{2}=0,\label{3herm1}\\
&&\hspace{-8.7cm}{\sf H}^{+}_{3}-{\sf H}^{+}_{1}-2p{\sf H}^{-}_{1} +2n{\sf H}^{-}_{3}=0,\label{3herm2}\\
&&\hspace{-8.7cm}{\sf H}^{+}_{3}-{\sf H}^{+}_{2}-2m{\sf H}^{-}_{2} +2n{\sf H}^{-}_{3}=0. \label{3herm3}
\end{eqnarray}
 Using \eqref{defintherm3} with \eqref{3herm1} the following identity is obtained:
\begin{eqnarray}
    &&\hspace{-2.3cm}(m+1) \left\lfloor \frac{m+n-p-1}{2} \right\rfloor ! \left\lfloor \frac{n+p-m+1}{2} \right\rfloor ! \left\lfloor \frac{m+p-n-1}{2} \right\rfloor ! \nonumber\\
    &&\hspace{-1.3cm} -(p+1) \left\lfloor \frac{m+n-p+1}{2} \right\rfloor !\left\lfloor \frac{n+p-m-1}{2} \right\rfloor ! \left\lfloor \frac{m+p-n-1}{2} \right\rfloor ! \nonumber\\
    &&\hspace{-1.3cm}  - \left\lfloor \frac{m+p-n+1}{2} \right\rfloor ! \left\lfloor \frac{m+n-p-1}{2} \right\rfloor ! \left\lfloor \frac{n+p-m+1}{2} \right\rfloor ! \nonumber\\
    && \hspace{-1.3cm} + \left\lfloor \frac{m+n-p+1}{2} \right\rfloor ! \left\lfloor \frac{m+p-n+1}{2} \right\rfloor ! \left\lfloor \frac{n+p-m-1}{2} \right\rfloor ! =0.
\end{eqnarray}
This can be rearranged to the following
\begin{eqnarray}
&&\hspace{1cm}\left( \frac{n+p-m+1}{2} \left( (m+1) - \frac{m+p-n+1}{2} \right) + \frac{m+n-p+1}{2} \left( -(p+1) + \frac{m+p-n+1}{2} \right) \right) \nonumber \\
&&\hspace{1.5cm} \times  \left \lfloor \frac{m+n-p-1}{2} \right \rfloor ! \left \lfloor \frac{n+p-m-1}{2} \right \rfloor ! \left \lfloor \frac{m+p-n-1}{2} \right \rfloor ! = 0.
\end{eqnarray}
Since the coefficient multiplying the factorials vanishes, we can see that the identity is trivially satisfied.
So this is a clear validation of the contiguous relations 
implied by Theorem \ref{THMAAA}.
Note that similar validations can be obtained by using \eqref{3herm2}, \eqref{3herm3}, which we leave to the reader.
\end{rem}

\subsection{{Linearization of a product of t}hree Hermite polynomials}
\begin{thm}
\label{thm5.3}
Let $p,m,n\in\N_0$ and without loss of generality $p\le m\le n$, 
$x\in\C$.
Then 
\begin{equation}
H_p(x)
H_m(x)
H_n(x)
=
\sum_{k=0}^{\lfloor\frac{p+m+n}{2}\rfloor}
{\sf  f}_{k,p,m,n} H_{p+m+n-2k}(x),
\label{prodthreeherm}
\end{equation}
where
\begin{equation}
{\sf  f}_{k,p,m,n}:=
\left\{
\begin{array}{lll}
{\sf  d}_{k,p,m,n}, \quad&\mbox{if}&\quad 0\le k \le p-1,\\[0.1cm]
{\sf  e}_{k,p,m,n}, \quad&\mbox{if}&\quad p\le k \le \lfloor \frac{p+m+n}{2}\rfloor,
\end{array}
\right.
\end{equation}
\begin{equation}
{\sf  d}_{k,p,m,n}\!:=\!
\frac{
(m\!+\!n)!p!2^k
}{
k!(p-k)!(m\!+\!n-k)!
}
\hyp{4}{3}{
{-k},
{-m},
{-n},
{-m\!-\!n\!+\!k},
}{
1\!+\!p\!-\!k,
\!\frac{\topt{0}{1}-m-n}{2}
}{\frac14}\!,
\end{equation}
\begin{eqnarray}
&&\hspace{-1.5cm}
{\sf  e}_{k,p,m,n}\!:=\!
\frac{
m!n!
(m\!+\!n\!+\!2p\!-\!2k)!2^k
}{
(k\!-\!p)!(m\!+\!p\!-\!k)!(n\!+\!p\!-\!k)!
(p\!+\!m\!+\!n\!-\!2k)!
}\nonumber\\
&&\hspace{2cm}\times
\hyp{4}{3}{
{-p},{-m\!-\!p\!+\!k},{-n\!-\!p\!+\!k},{-m\!-\!n\!-\!p\!+\!2k}\!,
}{
1\!+\!k\!-\!p,
\!\frac{\topt{0}{1}+2k-2p-m-n}{2}}{\frac14}\!.
\end{eqnarray}
\end{thm}
\begin{proof}
Using
(cf.~\cite[(9.8.34)]{Koekoeketal})
\[
H_n(x)=n! \lim_{\alpha\to\infty}\alpha^{-\frac{n}{2}}C_n^{\alpha}\left(\frac{x}{\sqrt{\alpha}}\right)
\]
four times in (\ref{prodthreegeg}), and \cite[(5.11.12)]{NIST:DLMF}, the result follows.
\end{proof}

\begin{rem}
Setting $p=0$ in (\ref{prodthreeherm}) straightforwardly produces (\ref{lin:herm}) 
since the first sum vanishes and the ${}_{4}F_{3}(\tfrac14)$ in the second sum is unity.
\end{rem}
\noindent Define the following definite integral of a product of four Hermite polynomials 
\[
{\sf H}(k,p,m,n):=
\int_{%
-\infty
}^\infty 
H_k(x)
H_p(x)
H_m(x)
H_n(x)
\,\expe^{-x^2}\,\dd x.
\]
Using orthogonality and the linearization relation
\eqref{prodthreeherm} one produces the following
corollary.
\begin{cor}
Let $k,p,m,n\in\N_0$. Then
\begin{eqnarray} \label{int4Herm}
\hspace{-2.3cm}{\sf H}(k,p,m,n)=
\left\{
\begin{array}{ll}
{\sf  E}_{k,p,m,n}, \quad&\mbox{if}\quad 0\le k \le m+n-p,\\[0.1cm]
{\sf  D}_{k,p,m,n}, \quad&\mbox{if}\quad m+n-p+2\le k \le m+n+p,\\[0.1cm]
0, \quad&\mbox{otherwise},
\end{array}
\right.
\end{eqnarray}
where 
\begin{eqnarray*}
&&\hspace{-0.3cm}{\sf  E}_{k,p,m,n}\!:=\!
\frac{\sqrt{\pi}\,
m!n!(p\!+\!k)!2^{\lfloor\frac{k+p+m+n}{2}\rfloor}
}{
\lfloor\frac{m+n-p-k}{2}\rfloor!
\lfloor\frac{k+p+n-m}{2}\rfloor!
\lfloor\frac{k+p+m-n}{2}\rfloor!
}
\hyp{4}{3}{
-k,
-p,
\lfloor\frac{n-m-p-k}{2}\rfloor,
\lfloor\frac{m-n-p-k}{2}\rfloor
}{
1\!+\lfloor\frac{m+n-p-k}{2}\rfloor,
\!\frac{\topt{0}{1}-k-p}{2}
}{\frac14}\!,
\end{eqnarray*}
\begin{eqnarray*}
&&\hspace{-0.2cm}{\sf  D}_{k,p,m,n}\!:=\!
\frac{\sqrt{\pi}\,
k!p!
(m\!+\!n)!2^{\lfloor\frac{p+k+m+n}{2}\rfloor}
}{
\lfloor\frac{k+p-m-n}{2}\rfloor!
\lfloor\frac{m+n+p-k}{2}\rfloor!
\lfloor\frac{m+n+k-p}{2}\rfloor!
}
\hyp{4}{3}{
-m,-n,
\lfloor\frac{p-k-m-n}{2}\rfloor,
\lfloor\frac{k-p-m-n}{2}\rfloor
}{
1\!+\lfloor\frac{p+k-m-n}{2}\rfloor,
\!\frac{\topt{0}{1}-m-n}{2}}{\frac14}\!.
\end{eqnarray*}
\end{cor}

\begin{proof}
Start with Theorem 
\ref{thm5.3}
and then integrate 
both sides using 
the property of orthogonality for Hermite polynomials
\eqref{orthogherm}.
This leads to
the following integral
for the linearization coefficients, namely
\begin{equation}
{\sf  f}_{k,p,m,n} = \frac{1}{h_{m+n+p-2k}} \int_{-\infty}^{\infty} H_p(x)H_m(x)H_n(x)H_{p+m+n-2k}(x) \,\expe^{-x^2} \dd x.
\end{equation}
This allows us to write the integral of a product of four Hermite polynomials as
\begin{equation}
{\sf H}(p,m,n,l)= h_{l}\, {\sf  f}_{\frac12 (p+m+n-l),p,m,n}. 
\end{equation}
\end{proof}

\begin{thm}
 Let $p,m,n\in\N_0$ and without loss of generality 
 $p \le m \le n$, $k \in \{0,...,p+m+n\}$.
Then the contiguous relations for the
integral for a product of four Hermite polynomials are given by
\begin{eqnarray}
&&\hspace{-3cm}{\sf  F}_{k,p,m+1,n} - {\sf  F}_{k,p+1,m,n} + 2m {\sf  F}_{k,p,m-1,n} - 2p {\sf  F}_{k,p-1,m,n} =0, \\
&&\hspace{-3cm}{\sf  F}_{k,p,m,n+1} - {\sf  F}_{k,p+1,m,n} + 2n {\sf  F}_{k,p,m,n-1} - 2p {\sf  F}_{k,p-1,m,n} =0, \\
&&\hspace{-3cm}{\sf  F}_{k+1,p,m,n} - {\sf  F}_{k,p+1,m,n} + 2k {\sf  F}_{k-1,p,m,n} - 2p {\sf  F}_{k,p-1,m,n} =0, \\
&&\hspace{-3cm}{\sf  F}_{k,p,m,n+1} - {\sf  F}_{k,p,m+1,n} + 2n {\sf  F}_{k,p,m,n-1} - 2m {\sf  F}_{k,p,m-1,n} =0, \\
&&\hspace{-3cm}{\sf  F}_{k+1,p,m,n} - {\sf  F}_{k,p,m+1,n} + 2k {\sf  F}_{k-1,p,m,n} - 2m {\sf  F}_{k,p,m-1,n} =0, \\
&&\hspace{-3cm}{\sf  F}_{k+1,p,m+1,n} - {\sf  F}_{k,p,m,n+1} + 2k {\sf  F}_{k-1,p,m,n} - 2n {\sf  F}_{k,p,m,n-1} =0, 
\end{eqnarray}
where
\begin{eqnarray}
{\sf  F}_{k,p,m,n} = 
\left\{
\begin{array}{ll}
{\sf  E}_{k,p,m,n}, \quad&\mbox{if}\quad 0\le k \le m+n-p,\\[0.1cm]
{\sf  D}_{k,p,m,n}, \quad&\mbox{if}\quad m+n-p+2\le k \le m+n+p,\\[0.1cm]
0, \quad&\mbox{otherwise},
\end{array}
\right.
\end{eqnarray}
\end{thm}
\begin{proof}
In addition to the contiguous functions shown in Remark \ref{rem5.2}, there are three additional ones:
\begin{eqnarray}
&&\hspace{-7.5cm}{\sf H}^{+}_{2}-{\sf H}^{+}_{1}-2p{\sf H}^{-}_{1} +2m{\sf H}^{-}_{2}=0,\\
&&\hspace{-7.5cm}{\sf H}^{+}_{3}-{\sf H}^{+}_{1}-2p{\sf H}^{-}_{1} +2n{\sf H}^{-}_{3}=0,\\
&&\hspace{-7.5cm}{\sf H}^{+}_{3}-{\sf H}^{+}_{2}-2m{\sf H}^{-}_{2} +2n{\sf H}^{-}_{3}=0.
\end{eqnarray}
Combining these with \eqref{int4Herm}, the contiguous relations of the theorem are obtained.
\end{proof}

\section{Jacobi polynomials}

The Jacobi polynomials can be defined  as \cite[(18.5.7)]{NIST:DLMF}
\begin{equation}
    P_n^{(\alpha,\beta)} = \frac{(\alpha+1)_n}{n!} \pFq{2}{1}{-n,n+\alpha+\beta+1}{\alpha +1}{\frac{1-x}{2}}.
\end{equation}
where $n \in \mathbb{N}_0$, $\alpha, \beta > -1$, $x \in \mathbb{C}$.  Jacobi polynomials are orthogonal on $x \in (-1,1),$ with the orthogonality relation \cite[(18.3)]{NIST:DLMF}
\begin{equation} \label{jacobi:ortho}
    \int_{-1}^1 P_m^{(\alpha,\beta)}(x) P_n^{(\alpha,\beta)}(x) (1-x)^\alpha (1+x)^\beta = \frac{2^{\alpha + \beta +1}\Gamma(n+\alpha+1)\Gamma(n+\beta+1)}{(2n+\alpha+\beta+1)\Gamma(n+\alpha+\beta+1)n!} \delta_{n,m} =: h_n \delta_{n,m}.
\end{equation}
The coefficients for the 
three-term recurrence relation for the Jacobi 
polynomials are given by  \cite[(18.9.2)]{NIST:DLMF}
\begin{eqnarray}\label{jacobi:3}
&&\hspace{-6.0cm}A_n := \frac{(2n + \alpha + \beta +1)(2n+ \alpha + \beta +2)}{2(n+1)(n + \alpha + \beta+1)}, \\
&&\hspace{-6.0cm}B_n := \frac{(\alpha^2 - \beta^2)(2n + \alpha + \beta+1)}{2(n+1)(n + \alpha + \beta +1)(2n + \alpha + \beta)}, \\
&&\hspace{-6.0cm}C_n := \frac{(n+\alpha)(n+\beta)(2n + \alpha +\beta+2)}{(n+1)(n+\alpha+\beta+1)(2n + \alpha+\beta)}.
\end{eqnarray}

\subsection{Linearization of a product
of two Jacobi polynomials}

The linearization formula for Jacobi polynomials was given in a 
fundamental work by Rahman \cite[(1.9)]{Rahman81a}. Rahman expressed the 
linearization coefficient in terms of a ${}_9F_8(1)$. His result was 
given in terms of some other variables $s=n-m$, $j=k-n+m$, and 
also contained a typographical error.
The typographical error was that the 
the term $(2s - 2n - \alpha - \beta)$ in Rahman's original 
publication should have been
written as $(2s - 2n - \alpha - \beta)_j$.  
({\it Note that in \cite[(2.1.1)]{Janssenetal}, 
it was realized that Rahman's result contained a typographical error.})
The corrected
and further simplified version of Rahman's result is  given as follows.

\begin{thm}\label{Rahman}{Rahman (1981).}
Let $m, n\in\N_0$ and without loss of generality, $n\ge m$, $\alpha,\beta\in\CC$.
Then
\begin{equation}
P_m^{(\alpha,\beta)}(x)
P_n^{(\alpha,\beta)}(x)
=\sum_{k=0}^{2m}
{\sf a}_{k,m,n}^{\alpha,\beta}
P_{k+n-m}^{(\alpha,\beta)}(x),
\label{lin:jac}
\end{equation}
where
\begin{eqnarray}
&&\hspace{-0.6cm}{\sf a}_{k,n,m}^{\alpha,\beta}:=
\frac{
(\alpha+1,\beta+1)_n(\alpha+\beta+1)_{2n-2m}(\alpha+\beta+1)_{2m}
(\alpha+\beta+1+2n-2m+2k)
}
{
m!(\alpha+\beta+1)_m(\alpha+1,\beta+1)_{n-m}(\alpha+\beta+2)_{2n}
(\alpha+\beta+1)
}
\nonumber\\
&&\hspace{-0.2cm}\times\frac{
(n-m+1,\alpha+\beta+2n-2m+1,2\alpha+2\beta+2n+2,-2m,
\alpha-\beta)_k
}
{k!(2\beta+2n-2m+2,\alpha+n-m+1,\alpha+\beta+2n+2,-\alpha-\beta-2m)_k
} \nonumber\\
&&\hspace{-0.2cm}\times\hyp98
{
\beta\!+\!n\!-\!m\!+\!\frac12,
\frac{\beta+n-m+\frac52}{2},
\beta\!+\!\frac12,
\beta\!+\!n\!+\!1,
-\alpha\!-\!m,
\frac{\alpha+\beta+k+\topt{1}{2}}{2}+n-m,
\frac{-k+\topt{0}{1}}{2}
}
{
\frac{\beta+n-m+\frac12}{2},
n\!-\!m\!+\!1,
\frac12\!-\!m,
\alpha\!+\!\beta\!+\!n\!+\!\frac{3}{2},
\frac{\beta-\alpha-k+\topt{1}{2}}{2},
\frac{k+\topt{2}{3}}{2}+\beta+n-m
}
{1}\!\!.
\end{eqnarray}
\end{thm}
\begin{proof}
See Rahman (1981) 
\cite[(1.9)]{Rahman81a} and 
\cite[(3.3)]{ReportOpenOPSFA13}
for description of correction.
\end{proof}

\begin{cor} \label{int3jacob}
Let $l,m,n \in \mathbb{N}_0$ and without loss of generality 
$m \le n$ and $l \in \{n-m,...,n+m\}$
using the linearization formula for Jacobi polynomials, it is possible to write the integral of the product of three Jacobi polynomials as follows:
\begin{eqnarray}
{\sf P}(l,m,n;\alpha,\beta) := \int_{-1}^1 P^{(\alpha, \beta)}_m(x) P^{(\alpha, \beta)}_n(x) P^{(\alpha, \beta)}_{l}(x) (1-x)^{\alpha} (1+x)^{\beta} \dd x = {\sf a}_{l+m-n,n,m}^{\alpha,\beta} h_{l},
\end{eqnarray}
where $h_l$ is defined in \eqref{jacobi:ortho}.
\end{cor}
\begin{thm}
Let $l,m,n \in \mathbb{N}_0$ and without loss of generality 
$m \le n$ and $l \in \{n-m,...,n+m\}$
\begin{eqnarray}
&&   \frac{ 2 (\alpha  - \beta) (\alpha + \beta)(m-n)(\alpha + \beta +m+n +1)(\alpha + \beta + 2n+1)}{(\alpha + \beta + 2m)(\alpha + \beta + 2n)} {\sf a}_{l+m-n,n,m}^{\alpha,\beta}  \nonumber \\
&&  \hspace{0.4cm}  + \frac{(\alpha + \beta +m+1)(m+1)(\alpha + \beta+2n+1)(\alpha + \beta + 2n + 2)}{(\alpha+\beta+2m+1)} {\sf a}_{l+m-n+1,n,m+1}^{\alpha,\beta} \nonumber \\
&& \hspace{0.4cm}   + \frac{(\alpha + m)(\beta +m)(\alpha + \beta + 2m+2)(\alpha + \beta +2n+1)(\alpha + \beta + 2n+2)}{(\alpha + \beta + 2m)(\alpha+\beta+2m+1)}
   {\sf a}_{l+m-n-1,n,m-1}^{\alpha,\beta} \nonumber \\
&& \hspace{0.4cm}   - (\alpha + \beta+2m+2)(\alpha+\beta +n+1)(n+1) {\sf a}_{l+m-n-1,n+1,m}^{\alpha,\beta} \nonumber \\
&& \hspace{0.4cm}   - \frac{(\alpha + \beta + 2m+2)(\alpha +n)(\beta+n)(\alpha + \beta +2n+2)}{(\alpha+\beta+2n)}{\sf a}_{l+m-n+1,n-1,m}^{\alpha,\beta} =0,
\end{eqnarray}
\begin{eqnarray}
&& \hspace{-2.4cm} \frac{ 2 (\alpha  - \beta) (\alpha + \beta)( l-n)(\alpha + \beta + n+l +1)}{(\alpha + \beta + 2 l)(\alpha + \beta + 2n)} 
{\sf a}_{l+m-n,n,m}^{\alpha,\beta}  
\nonumber \\
&& \hspace{-2.4cm} \hspace{0.4cm} - \frac{(\alpha + \beta + 2l + 2)(\alpha + \beta+n+1)(n+1)}{(\alpha + \beta + 2 n+1)} 
{\sf a}_{l+m-n-1,n+1,m}^{\alpha,\beta}  
\nonumber \\ 
&& \hspace{-2.4cm} \hspace{0.4cm}  - \frac{(\alpha +  n)(\beta + n)(\alpha + \beta + 2 l+2)(\alpha + \beta + 2n+2)}{(\alpha + \beta + 2 n+1)(\alpha + \beta + 2 n)}
{\sf a}_{l+m-n+1,n-1,m}^{\alpha,\beta}   
\nonumber \\
&& \hspace{-2.4cm} \hspace{0.4cm}  + \frac{(\alpha + \beta+2 n+2)(\alpha +l+1)(\beta + l+1)(\alpha + \beta +2l+1)}{(\alpha + \beta +2l+3)} 
{\sf a}_{l+m-n+1,n,m}^{\alpha,\beta}   
\nonumber \\
&& \hspace{-2.4cm} \hspace{0.4cm} + 
\frac{l(\alpha + \beta +l)(\alpha + \beta + 2 l+2)(\alpha + \beta + 2 n+2)}
{(\alpha + \beta + 2l-1)(\alpha + \beta + 2l)}
{\sf a}_{l+m-n-1,n,m}^{\alpha,\beta}   
=0,
\end{eqnarray}
\begin{eqnarray}
&& \hspace{-5.0cm} \frac{ 2 (\alpha  - \beta) (\alpha + \beta)(l-m)(\alpha + \beta +m+l +1)}{(\alpha + \beta + 2l)(\alpha + \beta + 2l+2)(\alpha + \beta + 2m)} 
{\sf a}_{l+m-n,n,m}^{\alpha,\beta}   
\nonumber\\
&& \hspace{-5.0cm}\hspace{0.4cm}  - 
\frac{(m+1)(\alpha+\beta+m+1)}
{(\alpha+\beta+2m+1)}
{\sf a}_{l+m-n+1,n,m+1}^{\alpha,\beta}  
\nonumber\\ 
&&\hspace{-5.0cm}\hspace{0.4cm}  - \frac{(\alpha + m)(\beta +m)(\alpha + \beta + 2m+2)}{(\alpha + \beta + 2m )(\alpha +\beta + 2m + 1)}
{\sf a}_{l+m-n-1,n,m-1}^{\alpha,\beta}  
\nonumber\\
&& \hspace{-5.0cm}\hspace{0.4cm} + \frac{(\alpha +l+1)(\beta + l+1)(\alpha + \beta +2m+2)}{(\alpha + \beta +2l+2)(\alpha+\beta +2l+3)} 
{\sf a}_{l+m-n+1,n,m}^{\alpha,\beta}  
\nonumber\\
&& \hspace{-5.0cm}\hspace{0.4cm} + \frac{l(\alpha + \beta+l)(\alpha + \beta + 2m + 2)}{(\alpha + \beta +2l-1)(\alpha + \beta +2l)} 
{\sf a}_{l+m-n-1,n,m}^{\alpha,\beta}  
=0.
\end{eqnarray}
\end{thm}
\begin{proof}
To determine the contiguous relations, the coefficients of the three-term recurrence relation defined in \eqref{jacobi:3} are used. 
Combining these coefficients with the general form of the contiguous relation of Theorem \ref{THMAAA} results in 
\begin{eqnarray*}
&&\hspace{-0.5cm}  \frac{ (\alpha  - \beta) (\alpha + \beta)(\alpha + \beta + 2n_j)(n_j-n_k)(\alpha + \beta +n_j+n_k +1)(\alpha + \beta + 2n_j+1)}{(n_j+1)(\alpha + \beta +n_j+1)(\alpha + \beta + 2n_j)(n_k+1)(\alpha + \beta + n_k+1)(\alpha + \beta + 2n_k)} {\sf P}  \\
&& \hspace{0.4cm} + \frac{(\alpha + \beta+2n_k+1)(\alpha + \beta + 2n_k + 2)}{2 (n_k+1)(\alpha + \beta +n_k+1)} {\sf P}_{n_j}^{+} \\
&& \hspace{0.4cm}   + \frac{(\alpha + n_j)(\beta +n_j)(\alpha + \beta + 2n_j+2)(\alpha + \beta +2n_k+1)(\alpha + \beta + 2n_k+2)}{2(n_j+1)(\alpha + \beta + n_k+1)(\alpha + \beta + 2n_j)(n_k+1)(\alpha +\beta + n_k+1)}
  {\sf P}_{n_j}^{-} \\
&&  = \frac{(\alpha + \beta + 2n_j+1)(\alpha + \beta+2n_j+2)(\alpha +n_k)(\beta + n_k)(\alpha + \beta +2n_k+2)}{2(n_j+1)(\alpha + \beta +n_j+1)(n_k+1)(\alpha+\beta +n_k+1)(\alpha + \beta +2n_k)} {\sf P}_{n_k}^{+} \\
&& \hspace{0.4cm} + 
\frac{(\alpha + \beta +2n_j+1)
(\alpha + \beta + 2n_j+2)}
{2(n_j+1)(\alpha + \beta +n_j+1)}
{\sf P}_{n_k}^{-}.
\end{eqnarray*}
Then using Corollary \ref{int3jacob} with the definition of the coefficient ${\sf a}_{k,n,m}^{\alpha,\beta}$ in Theorem \ref{Rahman} completes the proof.
\end{proof}
\section{Laguerre polynomials}
The Laguerre polynomials can be defined , $x\in\C$, $\alpha>-1$, \cite[Section 9.12]{Koekoeketal}
\[
L_n^\alpha(x)=\frac{(\alpha+1)_n}{n!}\hyp11{-n}{\alpha+1}{x},
\]
with orthogonality relation
\[
\int_0^\infty L_m^\alpha(x)L_n^\alpha(x)x^\alpha \,\expe^{-x}\,\dd x=\frac{\Gamma(\alpha+1+n)}{n!}\delta_{m,n}
=: h_m \delta_{m,n}.
\]
For Laguerre polynomials, one has (\ref{recur}),
\begin{equation}
A_n:=\frac{-1}{n+1},\quad
B_n:=\frac{2n+\alpha+1}{n+1},\quad
C_n:=\frac{n+\alpha}{n+1}.
\label{Laginfo}
\end{equation}
Define $n_1,\ldots,n_{N+1}\in\N_0$, $\Re\alpha>-1$,
\begin{equation}
{\sf L}({\bf n};\alpha):=\int_0^\infty
L_{n_1}^\alpha(x)
\cdots
L_{n_{N+1}}^\alpha(x)
x^\alpha\expe^{-x} \,\dd x.
\label{multintlag}
\end{equation}
Using the generating function for Laguerre polynomials \cite[(9.12.10)]{Koekoeketal}
\[
(1-t)^{-\alpha-1}\exp\left(-\frac{xt}{1-t}\right)=\sum_{n=0}^\infty L_n^\alpha(x)t^n,
\]
one finds that (\ref{multintlag}) has the following generating function
\cite[(9.3.7)]{Ismail:2009:CQO}
\begin{equation}
\sum_{n_1,\ldots,n_{k}=0}^\infty {\sf L}({\bf n};\alpha) t_1^{n_1}\cdots t_{k}^{n_{k}}
=\Gamma(\alpha+1)\left((1-t_1)\cdots(1-t_{k})\right)^{-\alpha+1}
\left(1+\sum_{j=1}^{k}\frac{t_j}{1-t_j}\right)^{-\alpha-1}.
\label{multlaggenfun}
\end{equation}

\subsection{{Linearization of a product of two Laguerre polynomials}}
\begin{thm}
Let $x\in\C$, $\alpha\in \mathbb C$, $n,m \in \No$, 
and without loss
of generality $n\ge m$. Then
\begin{align*}
\hspace{-1cm}L_m^{\alpha} (x) L_n^{\alpha} (x) &= 
\frac{
2^{2m}(\tfrac12)_m(\alpha\!+\!1)_n 
}{
(n\!-\!m)!
}\\
& \quad \times\sum_{k=n-m}^{n+m} 
\frac{
(-1)^{k+n+m}k!\,
}{
(\alpha\!+\!1)_k(m\!+\!n\!-\!k)!(k\!-\!n\!+\!m)!
}
\pFq{3}{2}{-m\!-\!\alpha,\frac{n-m-k+\topt{0}{1}}{2}}{n\!-\!m\!+\!1,\frac12\!-\!m}{1}
L_{k}^{\alpha}(x) 
.
\label{eq:thrm:linlag}
\end{align*}
\label{thrm:linlag}
\end{thm}
\begin{proof}
Let $x= 1-2x\beta^{-1}$ in (\ref{lin:jac}). Applying the relevant limiting relation \cite[(9.8.16)]{Koekoeketal} 
\begin{equation}
\lim_{\beta \to \infty} P_n^{(\alpha,\beta)} (1-2x\beta^{-1}) = L_n^{\alpha} (x),
\label{inter:jaclag}
\end{equation}
produces
\begin{equation*}
\lim_{\beta\to\infty} P_m^{(\alpha, \beta)} (1-2x\beta^{-1}) P_n^{(\alpha, \beta)} (1-2x\beta^{-1}) = \lim_{\beta\to\infty} \sum_{k=n-m}^{n+m}P_{k}^{(\alpha,\beta)}(1-2x\beta^{-1}) {\sf a}_{k,m,n}^{\alpha, \beta}.
\end{equation*}
One also has the following useful asymptotic result.
Let $k\in \No$, $b\in\C$, $c\in\C\setminus\{0\}$. Then
\begin{equation}
\lim_{a\to\infty} \frac{(ca+b)_k}{a^k} = c^k.
\label{lem:limpoch}
\end{equation}
Using (\ref{lem:limpoch}) to evaluate the limit on the right-hand side 
completes the proof. 
\end{proof}
Define $p,m,n\in\N_0$, $\alpha>-1$,
\[
{\sf L}_{p,m,n}^\alpha:=\int_0^\infty
L_p^\alpha(x)
L_m^\alpha(x)
L_n^\alpha(x)
x^\alpha\expe^{-x} \,\dd x,
\]
and using (\ref{fa}), (\ref{Laginfo}), we have for $|n-m|\le p\le m+n$,
\begin{align} \label{3LasF}
\hspace{-0.7cm}{\sf L}_{p,m,n}^{\alpha}
=
\frac{\Gamma(\alpha\!+\!1\!+\!m\!\vee\!n)(-1)^{p+n+m}2^{2(m\wedge n)}(\tfrac12)_{m\wedge n}}
{|n\!-\!m|!(m\!+\!n\!-\!p)!(p\!-\!|n-m|)!}
\pFq{3}{2}{-\alpha\!-\!m\!\wedge\!n,\frac{|n-m|-p+\topt{0}{1}}{2}
}{1\!+\!|n\!-\!m|,\frac12\!-\!m\!\wedge\!n}{1}.
\end{align}
Hence we have the following generating function (\ref{multlaggenfun})
\begin{align}
&\hspace{-0.6cm}\sum_{n,m=0}^\infty
\frac{ t_m^mt_n^n
2^{2(m\wedge n)}(\tfrac12)_{m\wedge n}(\alpha+1)_{m\vee n}
}
{|n-m|!}
\sum_{p=0}^{2(m\wedge n)}
\frac{(-1)^pt_p^{p+|n-m|}}
{(2(m\!\wedge\!n)-p)!\,p!}\pFq{3}{2}{-\alpha\!-\!m\!\wedge\!n,\frac{-p+\topt{0}{1}}{2}
}{1\!+\!|n\!-\!m|,\frac12\!-\!m\!\wedge\!n}{1}\nonumber\\
&=\left((1\!-\!t_p)(1\!-\!t_m)(1\!-\!t_n)\!+\!t_p(1\!-\!t_m)(1\!-\!t_n)\!+\!t_m(1\!-\!t_p)(1\!-\!t_n)\!+\!t_n(1\!-\!t_p)(1\!-\!t_m)\right)^{-\alpha-1}. 
\end{align}


\medskip
\begin{thm}
Let $p,m,n \in \mathbb{N}_0, a >1$ such that $|n-m| \le p \le m+n$. Then, 
one has three contiguous relations for the product of the three 
Laguerre polynomials given by
\small
\begin{align}
& \hspace{-0.9cm}
\frac{
(m\!+\!\alpha)(m\!+\!n\!-\!p)
\Gamma(\alpha\!+\!1\!+\!(m\!-\!1)\!\vee\!n)2^{2((m-1)\wedge n)}(\tfrac12)_{(m-1)\wedge n}}
{|n\!-\!m+1|!(p\!-\!|n\!-\!m\!+\!1|)!}
\pFq{3}{2}{-\alpha\!-\!(m\!-\!1)\!\wedge\!n,\frac{|n-m+1|-p+\topt{0}{1}}{2}
}{1\!+\!|n\!-\!m\!+\!1|,\frac12\!-\!(m\!+\!1)\!\wedge\!n}{1}\nonumber\\
& \hspace{-0.4cm}
-\frac{
(p\!+\!1)(m\!+\!n\!-\!p)
\Gamma(\alpha\!+\!1\!+\!m\!\vee\!n)2^{2(m\wedge n)}(\tfrac12)_{m\wedge n}}
{|n\!-\!m|!(p\!-\!|n-m|+1)!}
\pFq{3}{2}
{-\alpha\!-\!m\!\wedge\!n,\frac{|n-m|-p-1+\topt{0}{1}}{2}
}{1\!+\!|n\!-\!m|,\frac12\!-\!m\!\wedge\!n}{1}\nonumber\\
& \hspace{-0.4cm} +
\frac{
(m\!+\!1)
\Gamma(\alpha\!+\!1\!+\!(m\!+\!1)\!\vee\!n)2^{2((m+1)\wedge n)}(\tfrac12)_{(m+1)\wedge n}}
{
(m\!+\!n\!-\!p\!+\!1)
|n\!-\!m-1|!
(p\!-\!|n\!-\!m\!-\!1|)!}\pFq{3}{2}{-\alpha\!-\!(m\!+\!1)\!\wedge\!n,\frac{|n-m-1|-p+\topt{0}{1}}{2}
}{1\!+\!|n\!-\!m\!-\!1|,\frac12\!-\!(m\!+\!1)\!\wedge\!n}{1}\nonumber \\
& \hspace{-0.4cm} -
\frac{
(p\!+\!\alpha)
\Gamma(\alpha\!+\!1\!+\!m\!\vee\!n)2^{2(m\wedge n)}(\tfrac12)_{m\wedge n}}
{
(m\!+\!n\!-\!p\!+\!1)
|n\!-\!m|!
(p\!-\!|n\!-\!m\!-\!1|)!}\pFq{3}{2}{-\alpha\!-\!m\!\wedge\!n,\frac{|n-m|-p\!+\!1+\topt{0}{1}}{2}
}{1\!+\!|n\!-\!m|,\frac12\!-\!m\!\wedge\!n}{1}\nonumber\\
& \hspace{-0.4cm} -
\frac{
2(p\!-\!m)
\Gamma(\alpha\!+\!1\!+\!m\!\vee\!n)2^{2(m\wedge n)}(\tfrac12)_{m\wedge n}}
{|n\!-\!m|!(p\!-\!|n-m|)!}\pFq{3}{2}{-\alpha\!-\!m\!\wedge\!n,\frac{|n-m|-p+\topt{0}{1}}{2}
}{1\!+\!|n\!-\!m|,\frac12\!-\!m\!\wedge\!n}{1}=0,
\end{align}
\begin{align}
&\hspace{-0.9cm}
\frac{
(n\!+\!\alpha)(m\!+\!n\!-\!p)
\Gamma(\alpha\!+\!1\!+\!m\!\vee\!(n\!-\!1))4^{m\wedge(n-1)}(\tfrac12)_{m\wedge(n-1)}}
{|n\!-\!m-1|!(p\!-\!|n\!-\!m\!-\!1|)!}
\pFq{3}{2}{-\alpha\!-\!m\!\wedge\!(n\!-\!1),\frac{|n-m-1|-p+\topt{0}{1}}{2}
}{1\!+\!|n\!-\!m\!-\!1|,\frac12\!-\!m\!\wedge\!(n\!-\!1)}{1}\nonumber\\
&\hspace{-0.4cm} -\frac{
(p\!+\!1)(m\!+\!n\!-\!p)
\Gamma(\alpha\!+\!1\!+\!m\!\vee\!n)2^{2(m\wedge n)}(\tfrac12)_{m\wedge n}}
{|n\!-\!m|!(p\!-\!|n-m|+1)!}\pFq{3}{2}{-\alpha\!-\!m\!\wedge\!n,\frac{|n-m|-p-1+\topt{0}{1}}{2}
}{1\!+\!|n\!-\!m|,\frac12\!-\!m\!\wedge\!n}{1}\nonumber\\
&\hspace{-0.4cm}+
\frac{
}{
}
\frac{
(n\!+\!1)
\Gamma(\alpha\!+\!1\!+\!m\!\vee\!(n\!+\!1))2^{2(m\wedge(n+1))}(\tfrac12)_{m\wedge(n+1)}}
{
(m\!+\!n\!-\!p)
|n\!-\!m+1|!(p\!-\!|n\!-\!m\!+\!1|)!}\pFq{3}{2}{-\alpha\!-\!m\!\wedge\!(n\!+\!1),\frac{|n-m+1|-p+\topt{0}{1}}{2}
}{1\!+\!|n\!-\!m\!+\!1|,\frac12\!-\!m\!\wedge\!(n\!+\!1)}{1}\nonumber\\
&\hspace{-0.4cm}-
\frac{
(p\!+\!\alpha)
\Gamma(\alpha\!+\!1\!+\!m\!\vee\!n)2^{2(m\wedge n)}(\tfrac12)_{m\wedge n}}
{
(m\!+\!n\!-\!p)
|n\!-\!m|!(p\!-\!|n-m|+1)!}\pFq{3}{2}{-\alpha\!-\!m\!\wedge\!n,\frac{|n-m|-p+1+\topt{0}{1}}{2}
}{1\!+\!|n\!-\!m|,\frac12\!-\!m\!\wedge\!n}{1}\nonumber\\
&\hspace{-0.4cm}-
\frac{
2(p\!-\!n)
\Gamma(\alpha\!+\!1\!+\!m\!\vee\!n)2^{2(m\wedge n)}(\tfrac12)_{m\wedge n}}
{|n\!-\!m|!(p\!-\!|n-m|)!}\pFq{3}{2}{-\alpha\!-\!m\!\wedge\!n,\frac{|n-m|-p+\topt{0}{1}}{2}
}{1\!+\!|n\!-\!m|,\frac12\!-\!m\!\wedge\!n}{1}=0,
\end{align}
\begin{align}
& \hspace{-1cm}
\frac{
(n\!+\!\alpha)(m\!+\!n\!-\!p)
\Gamma(\alpha\!+\!1\!+\!(m\!-\!1)\!\vee\!n)4^{((m-1)\wedge n)}(\tfrac12)_{(m-1)\wedge n}}
{|n\!-\!m+1|!(p\!-\!|n\!-\!m\!+\!1|)!}
\pFq{3}{2}{-\alpha\!-\!(m\!-\!1)\!\wedge\!n,\frac{|n-m+1|-p+\topt{0}{1}}{2}
}{1\!+\!|n\!-\!m\!+\!1|,\frac12\!-\!(m\!+\!1)\!\wedge\!n}{1}\nonumber\\
&\hspace{-1cm} -
\frac{
(m\!+\!\alpha)(m\!+\!n\!-\!p)
\Gamma(\alpha\!+\!1\!+\!m\!\vee\!(n\!-\!1))4^{(m\wedge(n-1))}(\tfrac12)_{m\wedge(n-1)}}
{|n\!-\!m-1|!(p\!-\!|n\!-\!m\!-\!1|)!}
\pFq{3}{2}{-\alpha\!-\!m\!\wedge\!(n\!-\!1),\frac{|n-m-1|-p+\topt{0}{1}}{2}
}{1\!+\!|n\!-\!m\!-\!1|,\frac12\!-\!m\!\wedge\!(n\!-\!1)}{1}\nonumber\\
&\hspace{-0.4cm} -
\frac{(m\!+\!1)
\Gamma(\alpha\!+\!1\!+\!(m\!+\!1)\!\vee\!n)2^{2((m+1)\wedge n)}(\tfrac12)_{(m+1)\wedge n}}
{(m\!+\!n\!-\!p\!+\!1)
|n\!-\!m-1|!(p\!-\!|n\!-\!m\!+\!1|)!}\pFq{3}{2}{-\alpha\!-\!(m\!+\!1)\!\wedge\!n,\frac{|n-m-1|-p+\topt{0}{1}}{2}
}{1\!+\!|n\!-\!m\!-\!1|,\frac12\!-\!(m\!+\!1)\!\wedge\!n}{1}\nonumber\\
&\hspace{-0.4cm}+
\frac{(n\!+\!1)
\Gamma(\alpha\!+\!1\!+\!m\!\vee\!(n\!+\!1))2^{2(m\wedge(n+1))}(\tfrac12)_{m\wedge(n+1)}}
{(m\!+\!n\!-\!p\!+\!1)
|n\!-\!m+1|!(p\!-\!|n\!-\!m\!+\!1|)!}\pFq{3}{2}{-\alpha\!-\!m\!\wedge\!(n\!+\!1),\frac{|n-m+1|-p+\topt{0}{1}}{2}
}{1\!+\!|n\!-\!m\!+\!1|,\frac12\!-\!m\!\wedge\!(n\!+\!1)}{1}\nonumber\\
&\hspace{-0.4cm} -
\frac{
2(m\!-\!n)
\Gamma(\alpha\!+\!1\!+\!m\!\vee\!n)2^{2(m\wedge n)}(\tfrac12)_{m\wedge n}}
{|n\!-\!m|!(p\!-\!|n-m|)!}\pFq{3}{2}{-\alpha\!-\!m\!\wedge\!n,\frac{|n-m|-p+\topt{0}{1}}{2}
}{1\!+\!|n\!-\!m|,\frac12\!-\!m\!\wedge\!n}{1}=0.
\end{align}
\normalsize
\end{thm}{}
\begin{proof}
\noindent Using Theorem \ref{THMAAA}, one has
 the following three contiguous relations:
\begin{eqnarray} \label{Lcont}
&&\hspace{-3cm}(p+1){\sf L}^{+}_1-(m+1){\sf L}_2^{+} = (m+\alpha){\sf L}_2^{-} -(p+\alpha){\sf L}_1^{-} +2(p-m){\sf L} ,\\
&&\hspace{-3cm}(p+1){\sf L}^{+}_1 -(n+1){\sf L}_3^{+} = (n+\alpha){\sf L}_3^{-} -(p+\alpha){\sf L}_1^{-} +2(p-n){\sf L} ,\label{Lcont2}\\
&&\hspace{-3cm}(m+1){\sf L}^{+}_2 -(n+1){\sf L}_3^{+} = (n+\alpha){\sf L}_3^{-} -(m+\alpha){\sf L}_2^{-} +2(m-n){\sf L} .\label{Lcont3}
\end{eqnarray}
Using the linearization of a product of two Laguerre polynomials
(expressed as an integral of a product of three Laguerre polynomials) written as a hypergeometric function ${}_3F_2(1)$ \eqref{3LasF} and substituting it into
the above equations completes the proof.
\end{proof}
\subsection{{Linearization of a product of two scaled Laguerre polynomials}}
Consider the integral associated with linearization coefficients for a product of two scaled 
Laguerre polynomials
\[
{\cal L}_{p,m,n}^\alpha(a,b):=
\int_0^\infty 
L_p^\alpha(x)
L_m^\alpha(ax)
L_n^\alpha(bx)
x^\alpha\expe^{-x}\,\dd x,
\]
where $\Re\alpha>-1$ and ${\mathcal L}_{p,n,n}^\alpha(a,b)=0$ if $p\ge n+m+1$.

\begin{rem}
To see that $\Re \alpha > -1$ is true, the above integral can be written out in terms of hypergeometric series to become
\begin{align*}
    {\cal L}_{p,m,n}^\alpha(a,b) &=  
     \frac{(\alpha +1)_p (\alpha +1)_m (\alpha +1)_n}{p!m!n!} \\
     & \times \sum_{s=0}^p \sum_{k=0}^m \sum_{l=0}^n \frac{ (-p)_s (-m)_k (-n)_l a^k b^l }{(\alpha+1)_s (\alpha+1)_k (\alpha+1)_l s!k!l!} \int_0^\infty x^{s+k+l+\alpha} \expe^{-x} \dd x.
\end{align*}
The integral then becomes $\Gamma(\alpha + 1 +s +k+l)$, where it is required that $\Re (\alpha+s+k+l) > -1$.  Since this must be true for all values that $s$, $k$, and $l$ take in the summation, this requirement takes the form $\Re \alpha > -1$.
\end{rem}

We now present a theorem which
describes double sum 
representations for this
integral.

\begin{thm}
\label{scaleL}
Let $p,m,n\in\N_0$, $a,b>0$, $\alpha\in\C\setminus-\N$. Then 
\begin{eqnarray}
&&\hspace{-0.25cm}{\cal L}_{p,m,n}^\alpha(a,b)=
\frac{
\Gamma(\alpha+1) (\alpha+1)_n (-\alpha-p)_m
}{
p!m!(n+m-p)!
}
\left(-\frac{a}{b}\right)^m b^p
\nonumber\\
&&\hspace{1.5cm}\times\sum_{k=0}^\infty\sum_{l=0}^{\infty}
\frac{(-m)_k(-\alpha-m)_k(p+1)_l(\alpha+1+p)_l (p-n-m)_{k+l}}
{(\alpha+p-m+1)_{k+l}(p+1)_{k+l-m}k!l!}
\left(-\frac{b}{a}\right)^k b^l
,
\label{sLA}
\end{eqnarray}
\end{thm}
\begin{proof}
By inserting the definition of the Laguerre polynomials, this integral can
be converted to a triple sum, which can be evaluated using the integral
definition of the gamma function, namely
\begin{eqnarray*}
&&\hspace{-1.3cm}{\cal L}_{p,m,n}^\alpha(a,b)=
\frac{
(\alpha+1)_p
(\alpha+1)_m
(\alpha+1)_n
}{
p!m!n!
}\nonumber\\
&&\hspace{2cm}\times\sum_{s=0}^p
\sum_{k=0}^m
\sum_{l=0}^n
\frac{
(-p)_s
(-m)_k
(-n)_l
\Gamma(\alpha+1+s+k+l)
a^k
b^l
}{
(\alpha+1)_s
(\alpha+1)_k
(\alpha+1)_l\,
s!
k!
l!
}.
\end{eqnarray*}
By re-writing 
$\Gamma(\alpha+1+k+l+s)=\Gamma(\alpha+1)(\alpha+1)_{k+l}(\alpha+1+k+l)_{s}$,
one can write the sum over $s$ as a terminating Gauss hypergeometric series at unity.
The hypergeometric series can be evaluated using the Chu-Vandermonde identity 
\cite[(15.4.24)]{NIST:DLMF}. This converts the triple sum into a double sum which
produces
\begin{eqnarray}
&&\hspace{-2.0cm}{\cal L}_{p,m,n}^\alpha(a,b)=
\frac{\Gamma(\alpha+1+m)\Gamma(\alpha+1+n)}
{p!m!\,\Gamma(\alpha+1)}\nonumber\\
&&\hspace{0.5cm}\times\sum_{k=0}^m\sum_{l=0}^{n-p+k}
\frac{(-1)^{l+k}(-m)_k(p+1)_l(\alpha+1+p)_la^kb^{p+l-k}}
{(\alpha+1)_k(p+1)_{l-k}(\alpha+1+p)_{l-k}(n-p+k-l)!k!l!}.
\end{eqnarray}
In order to extend the sum indices to infinity, one can reverse
the order of the $k$ index by setting $k\mapsto m-k$.
Performing a series of standard manipulations for Pochhammer symbols produces the 
final form which completes the proof.
\end{proof}

\noindent An alternative form of the integral 
${\mathcal L}_{p,m,n}^\alpha(a,b)$ follows 
by utilizing the identity in Theorem \ref{scaleL}.
\begin{equation}
(p+1)_{r-m}=\frac{(p-m)!}{p!}(p-m+1)_r,
\label{critiden}
\end{equation}
which is valid for $p \ge m$.
\begin{cor}
\label{KDFL}
Let $p,m,n\in\N_0$, $p\ge m$, $a,b>0$, $\Re\alpha>-1$. Then 
\begin{equation}
{\cal L}_{p,m,n}^\alpha(a,b)
\label{sLA}
= 
\frac{ \Gamma(\alpha+1+n) (-\alpha-p)_m}{ m!  (p-m)! (n+m-p)!}
\left( -\frac{a}{b}\right)^m b^{\,p}
\,
Z_{p,m,n}^\alpha(a,b),
\end{equation}
where
\begin{eqnarray}
&&\hspace{-1cm}
Z_{p,m,n}^\alpha(a,b)
=
\sum_{l,k=0}^\infty
\frac{(p-n-m)_{k+l}
(-m,-\alpha-m)_k
(p+1,\alpha+1+p)_l}
{(p-m+1,\alpha+p-m+1)_{k+l}\,k!l!
}
{\left(-\frac{b}{a}\right)^k}
{b^l}\\
&&\hspace{0.98cm}=\KdF{2:0;0}{1:2;2}
{p-m+1,\alpha+p-m+1:-;-}
{p-n-m:-m,-\alpha-m;p+1,\alpha+p+1}{-\frac{b}{a},b},
\end{eqnarray}
where $F_{2:0;0}^{1:2;2}$ is terminating Kamp\'{e} de F\'{e}riet double hypergeometric series
\eqref{KdF}.  
\end{cor}
\begin{proof}
The terminating double hypergeometric Kamp\'{e} de F\'{e}riet form
\eqref{KdF}
of the integral ${\cal L}_{p,m,n}^\alpha(a,b)$
in Corollary \ref{KDFL}
follows from Theorem \ref{scaleL} by utilizing
the identity \eqref{critiden} 
($p \ge m$) with $r=l+l$ and \eqref{KdF}.
This completes the proof.
\end{proof}

\begin{thm}
\label{scaledLag}
Let $p,m,n \in \mathbb{N}_0$, $a,b, > 0$, $\alpha > -1$. Then the contiguous relations for the integral of the product of three Laguerre polynomials, which two are scaled is the following:
\begin{eqnarray}
&&\hspace{-0.4cm}a(p\!+\!\alpha){\cal L}_1^{-}\!-\!(m+\alpha){\cal L}_2^{-} \!+\!(2m\!+\!\alpha+1){\cal L}\!-\!a(2p+\alpha+1){\cal L}\!-\!(m+1){\cal L}_2^{+}\!+\!a(p+1){\cal L}_1^{+}\!=\!0,\\
&&\hspace{-0.4cm}b(p\!+\!\alpha){\cal L}_1^{-}\!-\!(n+\alpha){\cal L}_3^{-} \!+\!(2n\!+\!\alpha+1){\cal L}\!-\!b(2p+\alpha+1){\cal L}\!-\!(n+1){\cal L}_3^{+}\!+\!b(p+1){\cal L}_1^{+}\!=\!0,\\
&&\hspace{-0.4cm}b(m\!+\!\alpha){\cal L}_2^{-}\!-\!a(n+\alpha){\cal L}_3^{-}\!+\!a(2n\!+\!\alpha\!+\!1){\cal L}\!-\!b(2m\!+\!\alpha\!+\!1){\cal L}\!-\!a(n\!+\!1){\cal L}_3^{+}\!+\!b(m\!+\!1){\cal L}_2^{+}\!=\!0,
\end{eqnarray}
where
\begin{align}
{\cal L} & := {\cal L}(p,m,n;\alpha;a,b), \\
{\cal L}_1^\pm & := {\cal L}(p\pm1,m,n;\alpha;a,b),\\
{\cal L}_2^\pm & := {\cal L}(p,m\pm1,n;\alpha;a,b),\\
{\cal L}_3^\pm & := {\cal L}(p,m,n\pm1;\alpha;a,b).
\end{align}
\end{thm}
\begin{proof}
Again Theorem \ref{THMAAA} will be used with the coefficients of the three-term recurrence relations of Laguerre polynomials.  However, because we are now working with scaled Laguerre polynomials, one of these coefficients is now modified to the following:
\begin{equation}
A_n = - \frac{a}{n+1}.
\end{equation}
Where $a$ is the scaling factor of the Laguerre polynomial $L_n^\alpha(a x)$. The other two coefficients seen in \eqref{Laginfo} are unchanged.  This can be found by inserting the scaled Laguerre polynomial into the three-term recurrence relation and solving for the coefficients.\\
\end{proof}

\subsection*{Acknowledgements}
We would like to thank Roberto S.~Costas-Santos for valuable 
discussions.


\begin{thebibliography}{1}

\bibitem{ReportOpenOPSFA13}
R.~A. {Askey}, R.~K. {Beatson}, T.~S. {Chihara}, H.~S. {Cohl}, C.~F. {Dunkl},
  C.~{Koutschan}, S.~{Olver}, Y.~{Xu}, W.~{zu Castell}, and W.~{Zudilin}.
\newblock {Report from the Open Problems session at OPSFA13}.
\newblock {\em Symmetry, Integrability and Geometry: Methods and Applications
  (SIGMA), Special Issue on Orthogonal Polynomials, Special Functions and
  Applications}, 12:Paper 071, 12, 2016.

\bibitem{NIST:DLMF}
{\it NIST Digital Library of Mathematical Functions}.
\newblock \href{https://dlmf.nist.gov/}{\bf\tt\normalsize
  https://dlmf.nist.gov/}, Release 1.1.5 of 2022-03-15.
\newblock F.~W.~J. Olver, A.~B. {Olde Daalhuis}, D.~W. Lozier, B.~I. Schneider,
  R.~F. Boisvert, C.~W. Clark, B.~R. Miller, B.~V. Saunders, H.~S. Cohl, and
  M.~A. McClain, eds.

\bibitem{Ismail:2009:CQO}
M.~E.~H. Ismail.
\newblock {\em Classical and Quantum Orthogonal Polynomials in One Variable},
  volume~98 of {\em Encyclopedia of Mathematics and its Applications}.
\newblock Cambridge University Press, Cambridge, 2009.
\newblock With two chapters by Walter Van Assche, With a foreword by Richard A.
  Askey, Corrected reprint of the 2005 original.

\bibitem{IsmailKasraouiZeng2013}
M.~E.~H. Ismail, A.~Kasraoui, and J.~Zeng.
\newblock Separation of variables and combinatorics of linearization
  coefficients of orthogonal polynomials.
\newblock {\em Journal of Combinatorial Theory, Series A}, 120(3):561--599,
  2013.

\bibitem{Janssenetal}
A.~J. E.~M. Janssen, J.~J.~M. Braat, and P.~Dirksen.
\newblock {On the Computation of the Nijboer-Zernike Aberration Integrals at
  Arbitrary Defocus}.
\newblock {\em Journal of Modern Optics}, 51(5):687--703, 2004.

\bibitem{Koekoeketal}
R.~Koekoek, P.~A. Lesky, and R.~F. Swarttouw.
\newblock {\em Hypergeometric orthogonal polynomials and their
  {$q$}-analogues}.
\newblock Springer Monographs in Mathematics. Springer-Verlag, Berlin, 2010.
\newblock With a foreword by Tom H. Koornwinder.

\bibitem{Rahman81a}
M.~Rahman.
\newblock A nonnegative representation of the linearization coefficients of the
  product of {J}acobi polynomials.
\newblock {\em Canadian Journal of Mathematics. Journal Canadien de
  Math\'ematiques}, 33(4):915--928, 1981.

\bibitem{SriKarl}
H.~M. Srivastava and Per~W. Karlsson.
\newblock {\em Multiple {G}aussian hypergeometric series}.
\newblock Ellis Horwood Series: Mathematics and its Applications. Ellis Horwood
  Ltd., Chichester, 1985.

\end{thebibliography}

\def\cprime{$'$} \def\dbar{\leavevmode\hbox to 0pt{\hskip.2ex \accent"16\hss}d}

\end{document}